\documentclass[10pt, psfig]{siamlatex}
\usepackage{amssymb, amsmath} 

\newtheorem{stmt}{Theorem}
\newtheorem{cor}{ Corollary }
\newtheorem{lm}{Lemma}
\newtheorem{remark}{Remark}
\newcommand\be {\begin{equation}}
\newcommand\ee {\end{equation}}

\newcommand\bV {{\bf V}}

\newcommand\bx {{\bf x}}

\newcommand{\Real}{\mathbf{R}}

\def\DOT{\!\cdot\!}

\def\u{\mbox{\boldmath $u$}}

\def\f{\mbox{\boldmath $f$}}

\def\x{\mbox{\boldmath $x$}}

\def\Gmu {{G(\mu)}}
\def\Gnu {{G(\nu)}}
\def\Gnuk {{G(\nu k)}}
\def\dL2{{\dot{L}^2}}
\def\dH1{{\dot{H}^1_{per}}}
\def\om{\omega}

\def\omnm1{\|\omega^{n-1}\|}
\def\omnm1v{\|\nabla\omega^{n-1}\|}
\def\omnp1{\|\omega^{n+1}\|}
\def\omnp1v{\|nabla\omega^{n+1}\|}
\def\bSk{{\mathbb{S}_k}}
\def\mAk{{\mathcal{A}_k}}

\title{An efficient second order in time scheme for approximating long time statistical properties of the two dimensional Navier-Stokes equations}
\author{ Xiaoming Wang \thanks{Department of Mathematics, Florida State University, Tallahassee, FL 32306}({\tt wxm@math.fsu.edu})}
\date{\today}
\begin{document}

\maketitle

\begin{abstract}
We investigate the long tim behavior of the following efficient second order in time scheme  for the 2D Navier-Stokes equation in a periodic box:
$$    \frac{3\omega^{n+1}-4\omega^n+\omega^{n-1}}{2k} + \nabla^\perp(2\psi^n-\psi^{n-1})\cdot\nabla(2\omega^n-\omega^{n-1}) - \nu\Delta\omega^{n+1} = f^{n+1},
 \quad -\Delta \psi^n = \om^n.
$$
The scheme is a combination of a 2nd order in time backward-differentiation (BDF) and a special explicit Adams-Bashforth treatment of the advection term. Therefore only a linear constant coefficient Poisson type problem needs to be solved at each time step.
We prove uniform in time bounds on this scheme in $\dL2$,  $\dH1$ and $\dot{H}^2_{per}$ provided that the time-step is sufficiently small.
These time uniform estimates further lead to the convergence of long time statistics (stationary statistical properties)  of the scheme to that of the NSE itself  at vanishing time-step. Fully discrete schemes with either Galerkin Fourier or collocation Fourier spectral method are also discussed. 
\end{abstract}



\section{Introduction}

It is well-known that incompressible fluid flows could be extremely complex exhibiting seemingly random chaotic and/or turbulent behavior.
Statistical approaches are necessary in order to describe such kind of complex behavior (see for instance the treatises by Monin and Yaglom \cite{MY1975},   Frisch \cite{F1995},  Foias, Manley, Rosa and Temam \cite{FMRT2001}, Lasota and Mackey \cite{LM1994},  Majda and Wang \cite{MW2006} among others). 
If the the long time statistics of the system, i.e., the climate,  is the subject of study, we then need to investigate the  invariant measures of the system since it is the invariant measure (or stationary statistical solutions) that describes the long time statistics of the underlying dynamical system. 
Since most of these chaotic and/or turbulent systems are not amenable to analytical techniques at the present and in the near future, the issue of the development of numerical methods  that are able to capture the long time statistics becomes very important.
Higher order efficient schemes are apparently preferred due to the long time integration needed. 

In this paper, we will focus on the development and analysis of an efficient second order two step numerical method that is able to capture the long time statistics for the 
 following two dimensional  Navier-Stokes system for homogeneous incompressible Newtonian fluids in the  vorticity-streamfunction formulation (see for instance \cite{P2002})
 \begin{eqnarray}
   \frac{\partial\omega}{\partial t}+\nabla^\perp\psi\cdot\nabla\omega -\nu\Delta \omega &=& f, \label{NSE}
   \\
                                -\Delta\psi &=& \omega,
\end{eqnarray}
    where  $\omega$ denotes the vorticity, $\psi$ is the streamfunction, and $f$ represents (given) external body forcing, $\nu$ denotes the kinematic viscosity.  For simplicity we will assume periodic boundary condition, i.e., the domain is a two dimensional torus $\mathbb{T}^2=(0,2\pi)\times(0,2\pi)$, and that all functions are mean zero over the torus.
       
   For analytical data, it is known that the solution is analytic in space (in fact Gevrey class regular due to Foias and Temam \cite{FT1989}), and hence Fourier spectral is the obvious choice for spatial discretisation. As for time discretization, efficiency requires explicit treatment of the nonlinear term while stability calls for the implicit treatment of the diffusion term. Therefore we propose the following two step second order semi-implicit algorithm which treat the viscous term implicitly and the nonlinear advection term  explicitly
  \begin{eqnarray}
    \frac{3\omega^{n+1}-4\omega^n+\omega^{n-1}}{2k} + \nabla^\perp(2\psi^n-\psi^{n-1})\cdot\nabla(2\omega^n-\omega^{n-1}) - \nu\Delta\omega^{n+1} &=& f^{n+1},
 \label{scheme} \\
-\Delta\psi^j &=&\omega^j, j=n-1, n.
\end{eqnarray}
   Here $k$ is the time step, and $\omega^{n-1}, \omega^n, \omega^{n+1}$ are the approximation of the vorticity at  discrete time $(n-1)k, nk, (n+1)k$ respectively.  The convergence of this scheme on any fixed time interval can be derived via standard methods (see for instance \cite{MT1998}). There are many over-the-shelf efficient solvers of the (\ref{scheme}) since it essentially reduces to a Poisson solver  at each time step.
     This scheme falls into the category of the so-called implicit-explicit schemes (IMEX) \cite{ARW1995, Crouzeix1980} which combines  second order backward-differentiation (BDF) and a special second order Adams-Bashforth treatment of the nonlinear term. However, we would like to point out that our scheme is different from those classical ones where the Adams-Bashforth treatment of the nonlinear term is in the form of linear multistep fashion of $2\nabla^\perp\psi^n\cdot\nabla\omega^n-\nabla^\perp\psi^{n-1}\cdot\omega^{n-1}$ (see for instance Karniadakis, Israeli and Orszag \cite{ KIO1991}, or Ascher, Ruuth and Wetton \cite{ARW1995}, or Varah \cite{Varah1980}). 
     The classical one is also known as extrapolated Gear's  scheme.
     The new alternative treatment of the nonlinear advection term proves to be crucial in our long time stability analaysis.

   There is a long  list of work on time and spatial discretization of the NSE and related dissipative systems that preserves the dissipativity in various forms (see for instance \cite{FJKT1991, FJKT1994, Geveci1989, HillSuli1995, HillSuli2000, JKT1991, Ju2002, Larsson1989, MT1998, Shen1989, Shen1990, ToneW2006, Yan} among many others). In particular, a general framework for the convergence of global attractor (in upper semicontinous fashion) for one step scheme was derived by Hale, Lin and Raugel \cite{HLR1988}. The issue of the design of numerical schemes that can capture long time statistical behavior was clearly identified by Sigurgeirsson and Stuart \cite{SS2001}. Related issue in the case of Hamiltonian system was discussed in Tupper \cite{Tupper2005}.  It is  discovered recently by the author that if the dissipativity of dissipative system is preserved appropriately, then one step numerical scheme would be able to capture the long time statistical property of the underlying dissipative system asymptotically in the sense that invariant measures of the scheme would converge to those of the continuous in time system \cite{Wang2010}. This general framework for the convergence of long time statistics has been applied to the infinite Prandtl number model for convection by Cheng and Wang \cite{ChengWang2008}, the 2D Rayleigh-Benard convection by Tone and Wang \cite{ToneWang2010}, the 2D incompressible NSE by Gottlieb, Tone, Wang, Wang and Wirosoetisno \cite{GTWWW2011}. The idea of preserving certain properties of PDE in numerical discretisation is a well-known theme (see for instance \cite{LR2004} and the references therein for preserving the sympletic structure for Hamiltonian systems, \cite{TamWebb1993} for preserving dispersive relation in dispersive equations, as well as works cited above on preserving dissipativity for dissipative systems).
   
   The main purpose of this manuscript is to show,  as long as the time-step is sufficiently small, that the second order two step (three level) scheme (\ref{scheme}) is long time stable   in the sense that we are able to derive uniform in time estimates in various Sobolev spaces for the solutions to the scheme. We emphasize that the time-step restriction we have is independent of the spatial discetization although it depends on the data. Hence this is different from the usual CFL condition. We also show that the marginal distributions of the  invariant measures of the scheme converge to those of the NSE \eqref{NSE} at vanishing time-step. This may be viewed as an improvement of our earlier result on the convergence of long time statistical properties for a first order classical efficient scheme for the 2D Navier-Stokes equations \cite{GTWWW2011}.
   
  Multi-level numerical schemes are not discrete dynamical systems on the natural phase space. However, they can be naturally viewed as dynamical system on product space (see for instance the book by Stuart and Humphries \cite{SH1996}, or Hill and Suli's work on approximating global attractor of sectorial evolution equation via linear multistep methods \cite{HillSuli1995} among others).  However, the limit of this dynamical system on product space is not the product of the original (NSE) system, and certain projections must be used in order to put it into a framework that is similar to the semigroup set-up (Hill and Suli \cite{HillSuli1995} called it monoid).
Although the general framework for convergence of stationary statistical properties presented in \cite{Wang2010} may be modified to show the convergence of long time statistical properties using the monoid approach, here we use a more directly approach based on Liouville type equations as was done in the case of a first order scheme for the infinite Prandtl number model for convection, see Cheng and Wang \cite{ChengWang2008}. The limit invariant measure should be concentrated on the diagonal heuristically since the numerical method converges. We have to consider marginal distributions of the invariant measures of the scheme as a dynamical system on the product space so that it is compartible to the phase space of the NSE. We are able to show that all marginal distribution/measure of invariant measure of the scheme converge to an invariant measure/stationary stationary statistical property of the NSE \eqref{NSE} (see Wang \cite{Wang2008} for an application in terms of taking appropriate marginal distribution and convergence of stationary statistical properties in the context of infinite Prandtl number limit within the Boussinesq model for convection). Therefore long time statistics of the scheme (in the sense of generalized time average for instance) converges to those of the 2D NSE as the time-step shrinks to zero.
  
  The rest of the paper is organized as follows. We demonstrate the long time stability (boundedness) of the solution to the scheme in  $\dL2$ , $\dH1$ and $\dot{H}^2_{per}$ in section 2.  In section 3, we prove the convergence of the long time statistical properties.
  We briefly touch upon the issue of spatial discretization in section 4.  Final remarks and conclusions are offered at the end. An appendix covers two technical lemmas that are used in this manuscript. 

  \section{Time uniform bounds for the semi-discrete scheme}

  We first recall the well-known periodic Sobolev spaces on $\Omega=(0,2\pi)\times (0,2\pi)$ with average zero:
  \be
  \dot{H}^m_{per}(\Omega) := \left\{ \phi \in H^m(\Omega) \bigg| \int_\Omega \phi =0, \phi \ \text{periodic with period}\  2\pi \ \text{in each direction}\right\}
  \ee
  $\dot{H}^{-m}_{per}$ is defined as the dual space of $ \dot{H}^m_{per}$ with the duality induced by the $L^2$ inner product. We will use $\|\cdot\|:=\sqrt{\int_\Omega |\cdot|^2\,d\bx}$ to denote the $L^2(\Omega)$ norm.

    The adoption of $ \dot{H}^m_{per}$ is well-known (see for instance Constantin and Foias \cite{CF1988}, or Temam \cite{T1983}) since this space is invariant under the Navier-Stokes dynamics (\ref{NSE}) provided that the initial data and the forcing term belongs to the same space.

  \subsection{Well-posedness of the scheme}
  We first demonstrate that the scheme (\ref{scheme}) is well-posed in the space $L^2$.
\begin{lm}\label{wellpose}
For $f^{n+1}\in \dot{L}^2$, the scheme (\ref{scheme}) is well-posed on $L^2$ in the sense that there exists a unique solution in $L^2$ with $L^2$ data. Moreover, the following estimates hold:
\begin{eqnarray} 
2 \|\om^{n+1}\|^2+\nu k\|\nabla\om^{n+1}\|^2 
  &\le& \left(\frac{4\|\om^n\|+\|\om^{n-1}\|}{2}+ k\|f^{n+1}\|_{H^{-1}}\right)^2 + \frac{k C_w^2}{\nu}   (2\|\om^n\|+\|\om^{n-1}\|)^4, \forall n\ge 1,
 \label{L2H1}
\\
2 \|\nabla\om^{n+1}\|^2+\nu k\|\Delta\om^{n+1}\|^2 
  &\le& \left(\frac{4\|\nabla\om^n\|+\|\nabla\om^{n-1}\|}{2}+ k\|f^{n+1}\|\right)^2 
 \nonumber \\
&& + \frac{k C_w^2}{\nu}   (2\|\om^n\|+\|\om^{n-1}\|)^2(2\|\nabla\om^n\|+\|\nabla\om^{n-1}\|)^2, \forall n\ge 3.
 \label{H1H2}
\end{eqnarray}
In particular, $\om^n, \om^{n-1} \in \dL2$ implies $\om^{n+1}\in \dH1, \om^{n+2}\in \dot{H}^2_{per}$.
\end{lm}

\begin{proof}
  It is easy to see that for $\omega^n,  \omega^{n-1} \in \dot{L}^2$, we have $\psi^n,  \psi^{n-1}\in \dot{H}^2_{per}$. Hence 
\be \|\nabla^\perp(2\psi^n-\psi^{n-1})\cdot\nabla(2\omega^n-\omega^{n-1})\|_{\dot{H}^{-1}_{per}}
 \le C_w (2\|\om^n\|+\|\om^{n-1}\|)^2
\ee
 by the Wente type estimate (see proposition \ref{wente}).    Therefore, the scheme (\ref{scheme})  which can be viewed as a Poisson type problem
  \be
  \frac{3\omega^{n+1}}{2k} - \nu\Delta\omega^{n+1} =
     f ^{n+1}- \nabla^\perp(2\psi^n-\psi^{n-1})\cdot\nabla(2\omega^n-\omega^{n-1}) +\frac{4\omega^n-\omega^{n-1}}{2k}
     \in \dot{H}^{-1}_{per}
  \ee
   possesses a unique solution in  $\dot{L}^2$ (in fact in $\dot{H}^1_{per}$)  and the solution depends continuously on the data. Therefore it is well-posed on $\dot{L}^2$. The bounds on the solution follow from taking the inner product of the equation above with $\om^{n+1}$ and $-\Delta\om^{n+1}$ respectively, utilizing Cauchy-Schwarz as well as Wente type estimates (see Appendix).
   
\end{proof}

\subsection{Time uniform bound in $\dot{L}^2$}
  Now we derive the long time stability of the scheme (\ref{scheme})  in $\dot{L}^2$.

There are two approaches to derive the uniform in time bounds. The first one utlizes a generalized G-stability property of the 2nd order BDF due to Hill and Suli \cite{HillSuli2000}. The second one utlizies the original G-stability of the 2nd order BDF scheme \cite{HW2002} together with a novel two step discrete Gonwall type inequality (see lemma \ref{gronwall}). 

We first introduce the first approach based on a generalized G-stability of 2nd order BDF.
  
 \subsubsection{Generalized G-norm and G-stability identity}
  We will utilize the G-stability of the 2nd order backward differentiation in an essential way. More specifically, we introduce the following positive definite matrix for $\mu\ge 0$
  \begin{equation}
  G(\mu) = \left(\begin{array}{cc}
  		\frac12 & -1
		\\ 
		-1 & \frac52 + \frac{\mu}{2}
		 \end{array}
		 \right)
\label{G_matrix}
\end{equation}
and we introduce a family of  generalized G-norms on $\mathbb{R}^2$ as
\be
  |\bV |^2_\Gmu = \bV^T \Gmu \bV, \bV\in \mathbb{R}^2,
\ee
 and the associated norms on $\mathbb{R}^2$ valued functions as
 \be
   \|\bV \|^2_\Gmu = \int_{\Omega} \bV^T \Gmu \bV \,d\bx,   \bV(\bx) : \Omega\rightarrow \mathbb{R}^2.
\ee 
It is easy to see that $|\bV |_\Gmu$ is monotonically increasing in $\mu$, and is an equivalent norm on $\mathbb{R}^2$ for $\mu \in [0, 1]$ in the sense that there exist $0< C_l <1 < C_u 0$, such that for all $\mu\in [0,1]$,
\begin{eqnarray}
   C_l |\bV|^2_{G(0)}:=C_l |\bV|^2_{G} \le C_l |\bV |^2_\Gmu &\le& |\bV|^2 \le C_u |\bV |^2_{G(0)}:=C_u |\bV|^2_{G} \le C_u |\bV |^2_{G(\mu)}, \forall \bV \in \mathbb{R}^2 ,
  \label{G-equiv}
  \\
   C_l \|\bV\|^2_{G(0)}:=C_l \|\bV\|^2_{G} \le C_l \|\bV \|^2_\Gmu &\le& \|\bV\|^2_{L^2} \le C_u \|\bV \|^2_{G(0)}:=C_u \|\bV\|^2_{G} \le C_u \|\bV \|^2_{G(\mu)}, \forall \bV : \Omega\rightarrow \mathbb{R}^2 .
  \end{eqnarray}
Moreover, we have the following identity due to Hill and Suli \cite{HillSuli2000} (contained in the proof of their Lemma 6.1) which is a generalization of the G-stability of the BDF method (see for instance the classical book by Hairer and Wanner \cite{HW2002})
\be
 (\frac32 v_2-2 v_1+\frac12 v_0)v_2 +\frac{\mu}{2} v_2^2 
 =\frac12 \left(|{\bV}_1|^2_{G(\mu) } - \frac{1}{1+\mu} |{ \bV}_0|^2_{G(\mu)}\right)
 +\frac{((1+\mu)v_2-2v_1+v_0)^2}{4(1+\mu)}
 \label{G-identity}
\ee
  where $\bV_0=[v_0, v_1]^T, \bV_1=[v_1, v_2]^T \in \mathbb{R}^2$.
  This identity can be verified easily.

  \subsubsection{Time uniform bound in $\dL2$}
  
  We are now ready to show that  the scheme (\ref{scheme})  is uniformly bounded
in $L^2$, provided that the time step is sufficiently small.
 In
order to do so,
  we take the scalar product of (\ref{scheme}) with $2k\omega^{n+1}$ in $\dL2$ and utilize the 
  generalized G-stability of the 2nd order BDF scheme \eqref{G-identity} to obtain, with $\mu=\frac{\nu k}{2}$, 
\begin{eqnarray}
    && \|\bV_{n}\|^2_{\Gnuk} - \frac{1}{1+\nu k} \|{ \bV}_{n-1}\|^2_{G(\nu k) }
     	+\frac{\|(1+\nu k)\om^{n+1}-2\om^n+\om^{n-1}\|^2_{L^2}}{2(1+\nu k)}       
	  + 2 \nu k \| \nabla\omega ^{n+1}\|^2 -\nu k \|\om^{n+1}\|^2
	  \nonumber \\
	&&  + 2k \, b(2\psi^n-\psi^{n-1}, 2\omega^n-\om^{n-1}, \omega^{n+1})= 2  k (f^{n+1}, \omega ^{n+1}), n=1, 2, \cdots
	\label{1}
	\end{eqnarray}
	where $\bV_n = [\om^{n}, \om^{n+1}]^T$, and the trilinear term $b$ is defined as 
  \be \label{trilinear}
    b(\psi,\phi, \varphi)=\int_\Omega \nabla^\perp\psi\cdot\nabla\phi \, \varphi\,d\bx.
   \ee
	
Using the Cauchy--Schwarz type inequality,  and the equivalent norm on $\dH1$ that is determined by the $L^2$ norm of the gradient, , we majorize the right-hand side of (\ref{1}) by
\begin{equation}\label{2}
      2k \|f^{n+1}\|_{H^{-1}} \|\nabla\omega ^{n+1}\|
       \leq \frac{\nu k}{2} \|\nabla\omega ^{n+1}\|^2 + \frac{2k}{\nu }\|f^{n+1}\|^2_{H^{-1}}.
\end{equation}
Utilising 
the Wente type
estimate (\ref{158}), and the skew-symmetry of the trilinear term $b$ in the last two variables \eqref{trilinear}, we bound the nonlinear term as
\begin{eqnarray}\label{3}
 &&2k \, b(\nabla^\perp(2\psi^n-\psi^{n-1}), 2\omega^n-\om^{n-1}, \omega^{n+1})
 \nonumber \\
 &= &2k \, b(2\psi^n-\psi^{n-1}, -(1+\nu k)\om^{n+1}+ 2\omega^n-\om^{n-1}, \omega^{n+1}) 
  \nonumber \\
   &=& 2k \, b(2\psi^n-\psi^{n-1}, \omega^{n+1},  (1+\nu k)\om^{n+1}-2\omega^n+\om^{n-1}) 
\nonumber \\
& \le &
  2k C_w \|\nabla^\perp(2\psi^n-\psi^{n-1})\|_{H^1}\|\nabla\omega^{n+1}\| \|(1+\nu k)\om^{n+1}-2\omega^n+\om^{n-1} \| 
   \nonumber \\
&\le&
5k C_w \|\bV_{n-1}\|\|\nabla\omega^{n+1}\| \|(1+\nu k)\om^{n+1}-2\omega^n+\om^{n-1} \| 
   \nonumber \\
&\le&
 \frac{1}{4}\|(1+\nu k)\om^{n+1}-2\omega^n+\om^{n-1} \| ^2
 +25 k^2 C_w^2  \|\bV_{n-1}\|^2\|\nabla\omega^{n+1}\|^2 .
\end{eqnarray}
Relations (\ref{1})--(\ref{3}) imply, under the assumption that $\nu k\le 1$
\begin{eqnarray}
     \|\bV_{n}\|^2_{\Gnuk} - \frac{1}{1+\nu k} \|{ \bV}_{n-1}\|^2_{G(\nu k) }         
	 + (\frac{\nu}{2} -25C_w^2 k \|\bV_{n-1}\|^2) k \| \nabla\omega ^{n+1}\|^2 
	\le   \frac{2}{\nu}k \|f^{n+1}\|^2_{H^{-1}}.
	 \label{4}
	 \end{eqnarray}

We are now able to prove the following long time global energy stability result:
  \begin{stmt}[time uniform bound in $\dL2$]\label{t:bdh}
 Let $\omega^{n+1}$ be the solution of the numerical scheme (\ref{scheme}) and let  $f \in L^\infty(\Real_+;\ H^{-1})$, with $|f|_{\infty} := |f|_{L^\infty(\Real+; H^{-1})}$.
  Then there exists $M_{0k} = \max\{\|\bV_0\|_\Gnuk, \rho_0\}$ where $\rho_0=\frac{ 2|f|_{\infty}}{\nu}$ such that if the following time-step restriction is satisfied
  \begin{equation} \label{5a}
  k \leq \frac{\nu}{50C_w^2 C_u M_{0k}^2},
  \end{equation}
   then
 \be\label{q:bdinh}
  \frac{1}{\sqrt{C_u}}\|\bV_n\| \le  \|\bV_n\|_\Gnuk \leq M_{0k}, \, \forall \, n \geq 0,
\ee
  \be\label{q:bdv}
  \|\bV_n\|_\Gnuk^2 \le \frac{1}{(1+  {\nu}k)^n}\|\bV_0\|_\Gnuk^2
           +  \rho_0^2\left[ 1 - \frac{1}{( 1 + {\nu}k)^n}\right]
                ,
    \>\forall\, n\ge0.
 \ee
In particular, any ball in $(\dL2)^2$ of radius $\rho\ge \rho_0$ in the $\Gnuk$ norm, denoted $B_\Gnuk(\rho)$,  is invariant under the scheme.
\end{stmt}

\begin{proof}
The proof is straightforward  by induction on $n$. 
Indeed, utlising the fact that the $G$ norm is equivalent to classical norm \eqref{G-equiv}, we deduce from \eqref{4} that
\begin{eqnarray}
    \|\bV_{n}\|^2_{\Gnuk} - \frac{1}{1+\nu k} \|{ \bV}_{n-1}\|^2_{G(\nu k) }         
	 + (\frac{\nu}{2} -25C_w^2 C_u k \|\bV_{n-1}\|_\Gnuk^2) k \| \nabla\omega ^{n+1}\|^2 
	\le  \frac{\nu}{2}k \rho_0^2.
	 \end{eqnarray}
	 
	 It is clear
that (\ref{q:bdv}) and \eqref{q:bdinh} hold for  $n=0$. 

Assuming that (\ref{q:bdv}) and \eqref{q:bdinh}
hold for $n=0, \cdots, m$, we then have 
\be 
  \frac{\nu}{2} -25C_w^2 C_u k \|\bV_{m}\|_\Gnuk^2 \ge 0
 \ee
   and hence \eqref{q:bdv} remains valid for $n=m+1$ which further implies the validity of \eqref{q:bdinh} for $n=m+1$.

  We remark here that $M_{0k}$ could be replaced by a time-step $k$ independent quantity 
  \be
  M_0 = \max\{\|\bV_0\|_{G(\nu)}, \rho_0\}, \mbox{or} \ \tilde{M}_0=\max\{\frac{\|\bV_0\|}{\sqrt{C_l}}, \rho_0\}  
  \ee
   due to the monotonicity of the generalized G-norm in $\mu$ and the assumption that we consider time-step $k\le 1$.
 
 This completes the proof of theorem \ref{t:bdh}.
  \end{proof}

An immediate consequence of this theorem is the following absorbing property.
  \begin{cor}[absorbing property]\label{C1}
If the time-step is sufficiently small so that
 \be\label{q:k0}
   0 < k \leq \frac{\nu}{50C_w^2 C_u M_{0}^2}   =:k_0,
\ee 
then 
\be\label{q:tabs}
  \frac{1}{C_u}\|\bV_n\|^2 \le \|\bV_n\|_\Gnuk^2\leq 2 \rho_0^2,
        \quad \forall \, nk  \geq T_0(\|\bV_0\|_{\Gnu},|f|_\infty)
           :=\frac{4}{\nu } \ln\left(\frac{\|\bV_0\|_{\Gnu}}{\rho_0}\right).
\ee 
\end{cor}
\begin{proof}
The corollary is an easy consequence of the time uniform bound  \eqref{q:bdv}, equivalence of the generalized G-norms \eqref{G-equiv},   and the fact that $1+x \geq
\exp(x/2)$ if $x\in(0,1)$.
\end{proof}

 \medskip

\begin{remark}[alternative derivation of time uniform bound via discrete Gronwall type inequality with two steps] We now sketch the second approach of deriving uniform in time estimates based on the original G-stability of the 2nd order BDF scheme together with the two level generalized discrete Gronwall inequality presented in Lemma \ref{gronwall}.
We first notice that after taking the inner product of \eqref{scheme} with $\om^{n+1}$, applying the classical G-stability for 2nd order BDF scheme (corresponding to \eqref{G-identity} with $\mu=0$) , together with the same kind estimates utilized in the proof of Theorem \ref{t:bdh} leads to
\begin{eqnarray}
    \|\bV_{m}\|^2_{G} -  \|{ \bV}_{m-1}\|^2_{G} +\nu k \|\nabla\om^{m+1}\|^2        
	 + (\frac{\nu}{2} -25C_w^2 C_u  k \|\bV_{m-1}\|_G^2) k \| \nabla\omega ^{m+1}\|^2 
	\le  \frac{\nu}{2}k \rho_0^2.
  \label{G}
	 \end{eqnarray}
Multiplying \eqref{G} at $m=n$ by $\lambda\in (0,1)$ and add to \eqref{G} with $m=n+1$, we obtain, after utilizing the Poincar\'e inequality,  the equivalence of the G-norm \eqref{G-equiv}, and denoting $g^n=\|\bV_n\|^2_G, \varepsilon = \nu C_l\lambda k, \beta=\frac{(1+\lambda)\rho_0^2}{ C_l}$
\begin{eqnarray}
  (1+\varepsilon) g^{n+1} 
&\le& (1-\lambda) g^n + \lambda g^{n-1} + {\beta\varepsilon}
\\
 && - (\frac{\nu}{2} -25C_w^2  C_u k g^{n}) k \| \nabla\omega ^{n+2}\|^2
- \lambda(\frac{\nu}{2} -25C_w^2 C_u k g^{n-1}) k \| \nabla\omega ^{n+1}\|^2
\nonumber
	 \end{eqnarray}
	 where we have used the fact that
	 $$\nu k (\|\nabla\om^{n+2}\|^2 + \lambda \|\nabla\om^{n+1}\|^2) \ge \lambda\nu k \|\bV_{n+1}\|^2 \ge C_l \lambda\nu k \|\bV_{n+1}\|_G^2.$$
	 Now we assume that the following time-step restriction is satisfied:
\be 
  32 C_w^2\max\{g^n, g^{n-1}, 2\beta\}k \le {\nu}
\ee
It is then easy to see that if this time-step restriction is satisfied at $n$, then the same time step restriction is satisfied for all subsequent time thanks to Lemma \ref{gronwall} \eqref{Gstep1}. Therefore the assumption in Lemma \ref{gronwall} is valid for all $n$ provided it is valid initially.
Hence the uniform in time estimates follows from Lemma \ref{gronwall} with a long time bound independent of the initial data as a result of \eqref{Gfinal}.
The same approach works in deriving uniform bound in $\dH1$ and $\dot{H}^2_{per}$ as well when combined with the techniques from the next section.

A potential advantage of this alternative approach is the robustness of the inequality \eqref{G}.
\end{remark}

  \subsection{Time uniform bound  in $\dH1$ and $\dot{H}^2_{per}$}
  

 For the purpose of proving the uniform in time $\dH1$ estimates on the solution to the scheme \eqref{scheme}, we simply multiply the scheme by $-\Delta\om^{n+1}$ and utilize the generalized G-stability of the 2nd oder BDF scheme to obtain
\begin{eqnarray}
    && \|\nabla\bV_{n}\|^2_{\Gnuk} - \frac{1}{1+\nu k} \|{ \nabla\bV}_{n-1}\|^2_{G(\nu k) }
	  + 2 \nu k \| \Delta\omega ^{n+1}\|^2 -\nu k \|\nabla\om^{n+1}\|^2
	  \nonumber \\
	&&  - 2k \, b(2\psi^n-\psi^{n-1}, 2\omega^n-\om^{n-1}, \Delta\omega^{n+1})= -2  k (f^{n+1}, \Delta\omega ^{n+1}), n=1, 2, \cdots .
	\label{h1-1}
	\end{eqnarray}

The right-hand side of (\ref{h1-1}) can be majorized by
\begin{equation}
      2k \|f^{n+1}\| \|\Delta\omega ^{n+1}\|
       \leq \frac{\nu k}{2} \|\Delta\omega ^{n+1}\|^2 + \frac{2k}{\nu }\|f^{n+1}\|^2.
\end{equation}
Utilising 
the Wente type
estimate (\ref{158}), we bound the nonlinear term as
\begin{eqnarray}
 &&2k \, b(2\psi^n-\psi^{n-1}, 2\omega^n-\om^{n-1}, \Delta\omega^{n+1})
 \nonumber \\
& \le &
  2k C_w \|\Delta(2\psi^n-\psi^{n-1})\|\|\nabla(2\om^n-\omega^{n-1})\| \|\Delta\om^{n+1} \| 
   \nonumber \\
&\le&
5k C_w \|\bV_{n-1}\|(2\|\nabla\omega^{n+1}\| +\|\nabla((1+\nu k)\om^{n+1}-2\omega^n+\om^{n-1} )\| )\|\Delta\omega^{n+1}\|
   \nonumber \\
&\le&
 \frac{1}{4}\|\nabla((1+\nu k)\om^{n+1}-2\omega^n+\om^{n-1}) \| ^2
 +25 k^2 C_w^2  \|\bV_{n-1}\|^2\|\Delta\omega^{n+1}\|^2 
 +8kC_w\|\bV_{n-1}\|\|\om^{n+1}\|^\frac12\|\Delta\om^{n+1}\|^\frac32
    \nonumber \\
&\le&
 \frac{1}{4}\|\nabla((1+\nu k)\om^{n+1}-2\omega^n+\om^{n-1}) \| ^2
 +25 k^2 C_w^2  \|\bV_{n-1}\|^2\|\Delta\omega^{n+1}\|^2 
\nonumber \\
&&
 + \frac{\nu k}{8}\|\Delta\om^{n+1}\|^2 +\frac{Ck}{\nu^3}\|\bV_{n-1}\|^4\|\om^{n+1}\|^2
 .
\end{eqnarray}
Combining the inequalities above, under the assumption that $\nu k\le 1$, we obtain
\begin{eqnarray}
  &&   \|\nabla\bV_{n}\|^2_{\Gnuk} - \frac{1}{1+\nu k} \|{ \nabla\bV}_{n-1}\|^2_{G(\nu k) }
	 + (\frac{\nu}{2} -25C_w^2 k \|\bV_{n-1}\|^2) k \| \Delta\omega ^{n+1}\|^2 
	\nonumber \\
&\leq &   \frac{2}{\nu}k \|f^{n+1}\|^2 +\frac{Ck}{\nu^3}\|\bV_{n-1}\|^4\|\om^{n+1}\|^2.
	 \label{h1-2}
	 \end{eqnarray}

Similar to the derivation of theorem \ref{t:bdh}, we are able to show the following time uniform estimate in $\dH1$ and $\dot{H}^2_{per}$:
  \begin{stmt}[time uniform bound in $\dH1$ and $\dot{H}^2$]\label{t:bdv}
 Let $\omega^{n+1}$ be the solution of the numerical scheme (\ref{scheme}) and let  $f \in L^\infty(\Real_+;\ \dH1)$, with $\|f\|_{\infty} := \|f\|_{L^\infty(\Real+; \dL2)}, \|\nabla f\|_\infty:= \|\nabla f\|_{L^\infty(\Real+; \dL2)}$.
  Then there exist constants $\rho_1(\|\bV_0\|, \rho_0, \nu, \|f\|_\infty), \rho_2(\|\bV_0\|, \rho_0, \nu, \|\nabla   f\|_\infty)$ such that  if the time-step restriction \eqref{5a} is satisfied, we have, 
  \begin{eqnarray}\label{bdv}
  \|\nabla\bV_n\|_\Gnuk^2
& \le& \frac{1}{(1+  {\nu}k)^{n-2}}\|\nabla\bV_2\|_\Gnuk^2
           +  \rho^2_1(\|\bV_0\|, \rho_0, \nu, \|f\|_\infty)\left[ 1 - \frac{1}{( 1 + {\nu}k)^{n-2}}\right],
             n\ge 2 ,
\\
 \|\Delta\bV_n\|_\Gnuk^2
& \le& \frac{1}{(1+  {\nu}k)^{n-3}}\|\Delta\bV_3\|_\Gnuk^2
           +  \rho^2_2(\|\bV_0\|, \rho_0, \nu, \|\nabla f\|_\infty)\left[ 1 - \frac{1}{( 1 + {\nu}k)^{n-3}}\right],
          n\ge 3.
 \end{eqnarray}
Therefore, there exist integers $N_1(\|\bV_0\|_\Gnuk, \rho_0, \nu, \|f\|_\infty, k), N_2(\|\bV_0\|_\Gnuk, \rho_0, \nu, \|\nabla f\|_\infty, k)$ such that
\be
  \|\nabla\bV_n\|_\Gnuk \le \sqrt{2}\rho_1, \forall n \ge N_1,
\|\Delta\bV_n\|_\Gnuk \le \sqrt{2}\rho_2, \forall n \ge N_2.
\ee
\end{stmt}

\begin{proof}
The bound in $\dH1$ is essentially there already.  Indeed, if the time-step restriction \eqref{5a} is valid, we have
\begin{eqnarray*}
  \|\nabla\bV_{n}\|^2_{\Gnuk} - \frac{1}{1+\nu k} \|{ \nabla\bV}_{n-1}\|^2_{G(\nu k) }
&\leq &   \frac{2}{\nu}k \|f^{n+1}\|^2 +\frac{Ck}{\nu^3}\|\bV_{n-1}\|^4\|\om^{n+1}\|^2 
\\
&\le&  \frac{2}{\nu}k \|f^{n+1}\|^2 +\frac{Ck}{\nu^3}\left(\frac{1}{(1+\nu k)^{n-1}}\|\bV_{0}\|^6 + {\rho_0^6}\right)
\\
  && (\mbox{by Theorem \ref{t:bdh}, and the equivalence of G-norms \eqref{G-equiv}})
 \\
&\le& \nu k \rho_1^2 /2
\end{eqnarray*}
where
\be
  \rho_1^2 := \frac{4}{\nu^2}\|f\|^2_\infty + \frac{C}{\nu^4}(\|\bV_0\|^6+\rho_0^6).
\ee
We then deduce
$$ \|\nabla\bV_{n}\|^2_{\Gnuk} \le  \frac{1}{(1+\nu k)^2} \|{ \nabla\bV}_{2}\|^2_{G(\nu k) } + \rho_1^2. $$
The desired estimate $\dH1$ estimates then follows.

As for the $\dot{H}^2_{per}$ bound, we multiply \eqref{scheme} by $\Delta^2\om^{n+1}$ and utilize the following estimate on the nonlinear term:
\begin{eqnarray}
 &&2k \, b(2\psi^n-\psi^{n-1}, 2\omega^n-\om^{n-1}, \Delta^2\omega^{n+1})
 \nonumber \\
& \le &
  2k \, \|\nabla^\perp(2\psi^n-\psi^{n-1})\cdot\nabla(2\omega^n-\om^{n-1})\|_{H^1}\|\nabla \Delta\omega^{n+1}\|
 \nonumber \\
& \le &
  2k C_w \|\Delta(2\psi^n-\psi^{n-1})\|\|\Delta(2\om^n-\omega^{n-1})\| \|\nabla\Delta\om^{n+1} \| 
   \nonumber \\
&\le&
5k C_w \|\bV_{n-1}\|(2\|\Delta\omega^{n+1}\| +\|\Delta((1+\nu k)\om^{n+1}-2\omega^n+\om^{n-1} )\| )\|\nabla\Delta\omega^{n+1}\|
   \nonumber \\
&\le&
 \frac{1}{4}\|\Delta((1+\nu k)\om^{n+1}-2\omega^n+\om^{n-1}) \| ^2
 +25 k^2 C_w^2  \|\bV_{n-1}\|^2\|\nabla\Delta\omega^{n+1}\|^2 
 +8kC_w\|\bV_{n-1}\|\|\om^{n+1}\|^\frac13\|\nabla\Delta\om^{n+1}\|^\frac53
    \nonumber \\
&\le&
 \frac{1}{4}\|\Delta((1+\nu k)\om^{n+1}-2\omega^n+\om^{n-1}) \| ^2
 +25 k^2 C_w^2  \|\bV_{n-1}\|^2\|\nabla\Delta\omega^{n+1}\|^2 
\nonumber \\
 &&+ \frac{\nu k}{8}\|\nabla\Delta\om^{n+1}\|^2 +\frac{Ck}{\nu^5}\|\bV_{n-1}\|^6\|\om^{n+1}\|^2
 .
\end{eqnarray}
The rest of the proof are the same as those for the $\dH1$ estimate after we combine with Lemma \ref{wellpose} and notice that $\bV_3\in \dot{H}^2_{per}$ since $\bV_2\in \dH1$.  
 
 This completes the proof of theorem \ref{t:bdv}.
  \end{proof}
  %

  \section{Convergence of stationary statistical properties}
  
  The purpose of this section is to derive the main result of this paper, i.e.,  the convergence of long time stationary statistical properties of the scheme \eqref{scheme} to those of the Navier-Stokes system \eqref{NSE} as time-step approaches zero.

 \subsection{Dynamical system formulation}
We first recall the classical approach of writing the two step scheme \eqref{scheme} as a dynamical system on $(\dL2)^2$:
\be \label{scheme2}
  \bSk\left[\begin{array}{c} \om^{n}\\  \om^{n-1}\end{array}\right]
	=\left[\begin{array}{c} \om^{n+1}\\  \om^{n}\end{array}\right]
\ee
	thanks to Lemma \ref{wellpose}. 
Moreover, the rewritten scheme \eqref{scheme2} can be also viewed as a dynamical system on $B_\Gnuk(\rho), \forall \rho\ge\rho_0$ as long as the time-step restriction \eqref{5a} is satisfied thanks to Theorem \ref{t:bdh}. Furthermore, the trajectory is long time bounded in the $H^1, H^2$ norm independent of the initial data thanks to the time uniform estimates in $H^1, H^2$ that we derived in theorem \ref{t:bdv}. Therefore the dynamical system (\ref{scheme2}) possesses a unique global attractor $\mathcal{A}_k \subset B_\Gnuk(\rho_0)$ independent of $\rho$ (see for instance Temam \cite{T1997}) and uniformly bounded in $H^1, H^2$ (independent of $k$) provided that the external forcing term $f$ is time-independent.
  
Due to the dissipativity of the 2nd order BDF scheme, we are able to show that the two components of $\mathcal{A}_k$ are close in the sense that the distance is no more than a constant multiply of the time step $k$.
\begin{lm}[attractor bound and consistency]\label{l:diff}
Let $f\in \dH1$ be a time independent function.
 Then dynamical system \eqref{scheme2} is dissipative on $B_\Gnuk(\rho), \forall \rho \ge \rho_0$ as long as the time-step restriction \eqref{5a} is satisfied. Moreover, the attractors are uniformly bounded in the sense that for all $\bV\in\mAk$
\begin{eqnarray}
  \frac{1}{\sqrt{C_u}}\|\bV\| \le \|\bV\|_\Gnuk &\le & \rho_0,
\\ 
\frac{1}{\sqrt{C_u}}\|\nabla\bV\| \le \|\nabla\bV\|_\Gnuk &\le & \rho_1(\rho_0,\rho_0, \nu, \|f\|),
\\
\frac{1}{\sqrt{C_u}}\|\Delta\bV\| \le \|\Delta\bV\|_\Gnuk &\le & \rho_2(\rho_0, \rho_0,\nu, \|\nabla f\|).
\end{eqnarray}
Furthermore, we have the following consistency estimate in the sense that  there exists a constant $C_d$ independent of $k$ such that
  \be \label{diff}
  \|v_1-v_2\|+\sqrt{k}\|\nabla(v_1-v_2)\| \le C_d k, \forall \bV=[v_1,v_2]^T \in \mAk.
\ee
\end{lm}

\begin{proof}
  The uniform bound follows directly from Theorems \ref{t:bdh} and \ref{t:bdv}, and the equivalence of the G-norm \eqref{G-equiv}.

  As for the distance between the two coordinates, 
 it is straightforward from the scheme \eqref{scheme}, as well as the time uniform estimates derived in Theorem \ref{t:bdv} that are valid on the attractor $\mAk$, and the equivalence of the G-norm and the classical norm \eqref{G-equiv}, 
\begin{eqnarray*}
  \|\om^{n+1}-\om^n\|
  &\le& \frac13\|\om^n-\om^{n-1}\| + \frac23k(\|\Delta\om^{n+1}\|+\|f\|+\|\nabla^\perp(2\psi^n-\psi^{n-1})\cdot\nabla(2\om^n-\om^{n-1})\|)
\\
&\le& \frac13\|\om^n-\om^{n-1}\| + \frac23k(\|\Delta\om^{n+1}\|+\|f\|+5C_w\|\bV_{n-1}\|\|\nabla\bV_{n-1}\|)
\\
&\le& \frac13\|\om^n-\om^{n-1}\| + \frac23k(\|\Delta\bV_{n}\|_\Gnuk+\|f\|+5C_w C_u\|\bV_{n-1}\|_\Gnuk\|\nabla\bV_{n-1}\|_\Gnuk)
\\
&\le& \frac13\|\om^n-\om^{n-1}\| + \frac23k(\rho_2+\|f\|+5C_w C_u\rho_0\rho_1)
\\
&\le& \frac{1}{3^n}\|\om^1-\om^{0}\| + k(\rho_2+\|f\|+5C_w C_u\rho_0\rho_1)
\end{eqnarray*}
provided that $\bV_j=[\om^{j+1}, \om^j]^T, j=0, \cdots, n \in \mAk$.
Therefore 
\be
  \|(\bSk^n\bV_0)_2-(\bSk^n\bV_0)_1\| \le  \frac{1}{3^n}\|\om^1-\om^{0}\| + k(\rho_2+\|f\|+5C_w C_u\rho_0\rho_1)
\ee

On the other hand, thanks to the invariance of $\mAk$ under \eqref{scheme2}, for any $\bV\in\mAk$ and any integer $n$, there always exists a $\bV_0\in \mAk$ so that $\bSk^n\bV_0 = \bV$.
Letting $n$ approach infinity in the last inequality above, we deduce the desired  estimate on $v_1-v_2$.

The desired $H^1$ estimate follows from interpolating the time uniform $H^2$ estimates on the solution derived in Theorem \ref{t:bdv} and the $\dL2$ consistency estimate obtained above.

This ends the proof of the lemma.

\end{proof}

We remark that the distance between the two coordinates in the $\dH1$ norm can be improved to order $k$ as well.  This desired estimate in $H^1$ follows from a similar argument by multiplying the scheme \eqref{scheme} by $-\Delta(\om^{n+1}-\om^n)$ and perform the usual energy estimates. We leave the detail to the interested reader.

  
\subsection{Convergence of stationary statistical properties}

We first recall the notion of invariant measure and stationary statistical solution that characterizes the long time statistics of any given dynamical system. 

  We recall the definition of invariant measures.
\begin{definition} 
[Invariant measures]
A  Borel probability measure $\mu_k$ on $B_\Gnuk(\rho_0)$ is called an {\it
invariant measure} for $\bSk$ if
\be
  \int_{B_\Gnuk(\rho_0)} \Psi(\bSk(\bV)) d\mu_k=\int_{B_\Gnuk(\rho_0)} \Psi(\bV) d\mu_k
\ee
  for all bounded continuous test functional $\Psi$.

  The set of all invariant measures for $\bSk$ is denoted ${\cal{IM}}_k$.

We also recall that a Borel probability measure $\mu$ on $\dL2$ is
an {\it invariant measure, or stationary statistical solution},
for the Navier-Stokes system \eqref{NSE}, if
\begin{enumerate}
\item
    \be \int_{\dL2}
\|\nabla\om\|^2 \,d\mu(\om) < \infty,\ee
    \item
\be \int_{\dL2} <\nu\Delta\om-\nabla^\perp\psi\cdot\nabla\om+f, \Phi'(\om)>\,d\mu(\om)=0\ee
    for any cylindrical test functional $\Phi(\om)=\phi((\om, w_1), \cdots, (\om, w_m))$
    where $\phi$ is a $C^\infty$ function on $\mathbb{R}^m$, $\{w_j, j\ge1 \}$ being the standard Fourier basis which are also the eigenfunctions of $\Delta$
   in $\dL2$, and $<,>$ denotes the $H^{-1}$, $\dot{H}^1_{0,per}$ duality.
  \item
  \be
  \int_{\dL2}\int_\Omega\{\nu|\nabla\om|^2-f\om
     \}\,d{\bf x}\,d\mu(\om)\le 0.
  \ee
\end{enumerate}
    The set of all stationary statistical solutions for the
   Navier-Stokes system is denoted $\cal{IM}$.
\end{definition}

   Roughly speaking, the first condition says that the invariant
   measures are supported on the smaller and finer space of $\dH1$,
   the second condition is the differential form of the weak formulation of the invariance of
   the measure under the flow, and the third condition is a
   statistical version of the energy inequality.

We also recall the well-known fact that the Navier-Stokes system \eqref{NSE} generates a dissipative dynamical system on $\dL2$ with an absorbing ball of radius $\rho_0$ (see for instance \cite{T1997, CF1988}). Therefore, any invariant measure or stationary statistical solutions of \eqref{NSE} must be supported in $B_\dL2(\rho_0)$ in $\dL2$. Consequently the support of any invariant measure or stationary statistical solution to \eqref{NSE} must be supported on $B_\dL2(\rho_0)$ since all invariant measures are supported on the global attractor (see for instance \cite{FMRT2001, Wang2009}). Henceforth, we only need to take test functionals (observables) $\Phi$ to be supported on $B_\dL2(\rho_0)$.

Next, we observe that the invariant measures of the dynamical system defined by our numerical scheme \eqref{scheme2} and our Navier-Stokes system \eqref{NSE} do not share the same phase space. Moreover, a simple extension of the NSE to the product space is not appropriate, Therefore we should look at convergence of marginal distributions of the stationary statistical solutions to our scheme \eqref{scheme2}. Since the attractors of the scheme \eqref{scheme2} converge to the diagonal in the product space at vanishing step-size due to the consistency estimates \eqref{diff}, we anticipate that all marginal measures of the invariant measures of \eqref{scheme2} would converge to some invariant measure (stationary statistical solution) to the NSE \eqref{NSE}. Hence we anticipate the following
result which is the main finding of this manuscript.

\begin{stmt}[convergence of stationary statistical properties] \label{conv-IM}
Let $f\in \dH1$ be a time-independent function. 
Then the discrete dynamical systems defined via the scheme \eqref{scheme2} is autonomous and dissipative with non-empty set of invariant measures $\mathcal{IM}_k$. 
Denote $\mathcal{P}_j, j=1,2$  the projection from $(\dL2)^2$ onto its $j^{th}$ coordinate.
Let
$\{\mu_k, k\in (0, k_0]\}$ with $\mu_k\in \mathcal{IM}_k, \forall k$, be an arbitrary invariant measure of the numerical scheme \eqref{scheme2}.
Then each subsequence of $\{\mu_k\}$ must contain a subsubsequence (still
denoted $\{ \mu_k\}$) and two invariant measure $\mu_j$ of the
NSE \eqref{NSE} so that $\mathcal{P}_j\mu_k$ weakly converges to
$\mu_j$, i.e.,
\be
  \mathcal{P}_j^*\mu_k \rightharpoonup \mu_j, k\rightarrow 0 ,
\ee
  where $\mathcal{P}_j^*\mu_k (S) = \mu_k(\mathcal{P}_j^{-1}(S)), \forall S \in \mathcal{B}(\dL2)$.
\end{stmt}

\begin{proof}
The non-emptiness of the set of the invariant measures $\mathcal{IM}_k$ follows from the 
  uniform estimates in $H^1$ on the global attractors of the scheme \eqref{scheme2} proved in Lemma \ref{diff}, as well as the classical Bogliubov-Krylov technique (see for instance \cite{FMRT2001, VF1988}).
  
We observe that the family of Borel measures $\{\mu_k, k\in (0, k_0]\}$ is tight in the space of probability measures on $(\dL2)^2$ (after trivial extension to the outside of the absorbing ball) (see for instance Billingsley \cite{B1971}, or Lax \cite{L2002}). Therefore, for any subsequence of $\{\mu_k, k\in (0, k_0]\}$, there exists a subsubsequence, still denoted $\{\mu_k, k\in (0, k_0]\}$, and a Borel probability measure $\mu_0$ on $(\dL2)^2$ such that $\mu_k$ weakly converges to $\mu_0$. Our goal here is to show that the marginal measures of $\mu_0$  are in fact stationary statistical solutions of the Navier-Stokes system \eqref{NSE}, i.e.  $\mathcal{P}_j^*\mu_0\in \mathcal{IM}$. We work on the case of $j=1$ without loss of generality. The proof is similar to those presented in Cheng and Wang \cite{ChengWang2008} by utilizing the differential form of the invariant measure (stationary statistical solution).

 The first condition in the definition is easily verified since
  the global attractors for the discrete dynamical systems \eqref{scheme2} are
  uniformly bounded in $\dH1$ independent of the time step $k$ thanks to Lemma \ref{diff}, and the fact that the invariant measures are
  supported on the global attractor \cite{FMRT2001, Wang2009}.

  In order to check the second condition, i.e., the differential form of the weak formulation of invariance,
we let $\Phi(\om)=\phi((\om, w_1), \cdots, (\om,
w_m))=\phi(z_1, \cdots, z_m)$ be a smooth cylindrical test functional.
Notice that
\be
  \Phi'(\om)=\sum_{j=1}^m \frac{\partial}{\partial z_j}\phi((\om, w_1), \cdots, (\om,
  w_m))w_j,
\ee
  hence, denoting $<, >$ the duality between $H^{-1}$ and
  $\dH1$, we have
\begin{eqnarray*}
&& \int_{\dL2} \left< \nu\Delta\om-\nabla^\perp\psi\cdot\nabla\omega+f,
     \Phi'(\om)\right>\,d\mathcal{P}_1^*\mu_0(\om)
\\
&=& 
 \int_{(\dL2)^2} \left< \nu\Delta v_1-\nabla^\perp\psi_1\cdot\nabla v_1+f,
     \Phi'(v_1)\right>\,d\mu_0(\bV) \quad (\bV=[v_1, v_2]^T, -\Delta\psi_j=v_j, j=1,2)
\\
  && (\mbox{by definition of marginal measure})
\\
&=& 
\int_{(\dL2)^2} \sum_{j=1}^m \frac{\partial\phi}{\partial
z_j}\int_\Omega \left(\nu v_1\Delta w_j +\nabla^\perp\psi_1\cdot\nabla w_j v_1 + f w_j\right)\,d{\bf x}\,d\mu_0(\bV)
\\
  && (\mbox{by the choice of the cylindrical test functional $\Phi$ and integrations by parts})
\\
&=& \lim_{k\rightarrow 0}
\int_{(\dL2)^2} \sum_{j=1}^m \frac{\partial\phi}{\partial
z_j}\int_\Omega \left(\nu v_1\Delta w_j +\nabla^\perp\psi_1\cdot\nabla w_j v_1 + f w_j\right)\,d{\bf x}\,d\mu_k(\bV)
\\
  && (\mbox{by the definition of weak convergence})
\\
&=& \lim_{k\rightarrow 0}
\int_{(\dL2)^2} \sum_{j=1}^m \frac{\partial\phi}{\partial
z_j}\int_\Omega \left(\nu\bSk(\bV)_2\Delta w_j +\nabla^\perp(2\psi_2-\psi_1)\cdot\nabla w_j (2v_2-v_1) + f w_j\right)\,d{\bf x}\,d\mu_k(\bV)
\\
  && (\mbox{by the consistency estimate \eqref{diff} as well as the choice of the cylindrical test functional})
\\
&=&
 \lim_{k\rightarrow 0}
\int_{(\dL2)^2} \left<\nu\Delta\bSk(\bV)_2-\nabla^\perp(2\psi_2-\psi_1)\cdot\nabla (2v_2-v_1) + f, \Phi'(v_1)\right>\,d\mu_k(\bV)
\\
  && (\mbox{by the choice of the cylindrical  test functional})
\\
&=&
 \lim_{k\rightarrow 0}
\int_{(\dL2)^2} \left<\nu\Delta\bSk(\bV)_2-\nabla^\perp(2\psi_2-\psi_1)\cdot\nabla (2v_2-v_1) + f, \Phi'(\frac32v_2-\frac12v_1)\right>\,d\mu_k(\bV)
\\
  && (\mbox{by the consistency estimate \eqref{diff} as well as the choice of the cylindrical test functional})
\\
&=&  \lim_{k\rightarrow 0}
\int_{(\dL2)^2} \left<[-\frac12,\frac{3}{2}]^T \cdot \frac{\bSk(\bV)-\bV}{k}, \Phi'(\frac32v_2-\frac12 v_1)\right>\,d\mu_k(\bV)
\\
  && (\mbox{according to the scheme \eqref{scheme}, \eqref{scheme2}}) 
\\
&=& \lim_{k\rightarrow 0}\int_{(\dL2)^2} \frac{1}{k}
     (\Phi(\bSk(\bV)\cdot[-\frac12,\frac32]^T)-\Phi(\bV\cdot[-\frac12,\frac32]^T))\,d\mu_k(\bV)
\\
	&&(\mbox{by calculus and Lemma \ref{diff}})
\\
&=& 0
\\
&& (\mbox{by the invariance of }\ \mu_k \ \mbox{under}\ \bSk \ \mbox{applied to the test functional}\ \Phi(\bV\cdot[-\frac12,\frac32]^T))
\end{eqnarray*}
where we have used the boundedness and continuity of
$\frac{\partial\phi}{\partial z_j}$ on the union of the support of
$\mu_k$, the consistency estimate (\ref{diff}), the invariance
of $\mu_k$ under $\bSk$.

 This proves the differential form of the weak invariance of $\mathcal{P}_1^*\mu_0$ under the 2D Navier-Stokes dynamics, i.e., no. 2.

  The energy type
inequality no.3 can be verified easily as well. For this purpose,
we first show that any invariant measure $\mu_k$ of the numerical
scheme (\ref{scheme}) must satisfy an energy type estimate.
The desired continuous one will be the limit as the time step
approaches zero.

Taking the inner product of the scheme
(\ref{scheme}) with $\om^{n+1}$ and we have
\begin{eqnarray*}
 &&   \frac{1}{2k}(\|\bV^{n}\|_G^2 - \|\bV^{n-1}\|_G^2
+\|\om^{n+1} -2\om^n
  -\om^{n-1}\|^2)
  +\nu\|\nabla\om^{n+1}\|^2
\\
&&
  +\int_\Omega (\nabla^\perp(2\psi^n-\psi^{n-1})\cdot\nabla(2\om^n-\om^{n-1})\om^{n+1}
  -f \om^{n+1})=0 .
\end{eqnarray*}
  This can be re-written using the discrete dynamical system
  notation $\bSk$ as
\begin{eqnarray*}
&&    \frac{1}{2k}(\|\bSk(\bV)\|_G^2 - \|\bV\|_G^2
+\|(\bSk(\bV) - \bV)\cdot [1,-1]^T\|^2)
  +\nu\|\nabla (\bSk(\bV))_2\|^2
\\
&&
  +\int_\Omega  (\nabla^\perp(2\psi_2-\psi_{1})\cdot\nabla(2v_2-v_1)-f) (\bSk(\bV))_2 =0 .
\end{eqnarray*}

  Integrating this identity with respect to the invariant measure $\mu_k$ and letting $k\rightarrow 0$ we have
\begin{eqnarray*}
 0  &\ge &  -\liminf_{k\rightarrow 0}\frac{1}{2k} \int_{(\dL2)^2}\|(\bSk(\bV) - \bV)\cdot [1,-1]^T\|^2 d\mu_k(\bV)
\\
&=&  \liminf_{k\rightarrow 0}
\int_{(\dL2)^2}\int_\Omega (\nu|\nabla  (\bSk(\bV))_2|^2
 +(\nabla^\perp(2\psi_2-\psi_{1})\cdot\nabla(2v_2-v_1)-f) (\bSk(\bV))_2)\,d\bx\,d\mu_k(\bV)
\\
&\ge&\int_{(\dL2)^2}\int_\Omega \left(\nu|\nabla v_1|^2
 +(\nabla^\perp\psi_{1}\cdot\nabla v_1-f) v_1\right)\,d\bx\,d\mu_0(\bV)
\\
&=&\int_{\dL2}\int_\Omega\{\nu|\nabla\om|^2-f\om
     \}\,d{\bf x}\,d\mathcal{P}_1^*\mu_0(\om),
\end{eqnarray*}
where we have utilized the invariance of $\mu_k$ under $\bSk$ and the following estimates
\begin{eqnarray*}
    \int_{(\dL2)^2} \|\nabla (\bSk(\bV))_2 \|^2 \,d\mu_k(\bV)
 &=& \int_{(\dL2)^2} \lim_{m\rightarrow\infty}
  \sum_{j=1}^m   \frac{(\nabla  (\bSk(\bV))_2, \nabla w_j)^2}{\|\nabla
  w_j\|^2}  \,d\mu_k(\bV)
\\
  && (\mbox{since} \ \{w_j\} \ \mbox{form an orthogonal basis in}\ \dH1)
\\
&=& \lim_{m\rightarrow\infty}
  \sum_{j=1}^m  \int_{(\dL2)^2} \frac{(\nabla  (\bSk(\bV))_2, \nabla w_j)^2}{\|\nabla
  w_j\|^2}  \,d\mu_k(\bV)
\\
  && (\mbox{Lebesque dominated convergence theorem})
\\
 &=& \lim_{m\rightarrow\infty}
  \sum_{j=1}^m  \int_{(\dL2)^2} \frac{(  (\bSk(\bV))_2, \Delta w_j)^2}{\|\nabla
  w_j\|^2}  \,d\mu_k(\bV)
\\
&& (\mbox{integration by parts})
\\
 &=& \lim_{m\rightarrow\infty}
  \sum_{j=1}^m  \int_{(\dL2)^2} \frac{(  v_2, \Delta w_j)^2}{\|\nabla
  w_j\|^2}  \,d\mu_k(\bV)
\\
&& (\mbox{invariance of $\mu_k$ under $\bSk$})
\\
 &=& \lim_{m\rightarrow\infty}
  \sum_{j=1}^m  \int_{(\dL2)^2} \frac{(  (\bSk(\bV))_1, \Delta w_j)^2}{\|\nabla
  w_j\|^2}  \,d\mu_k(\bV)
\\
&& (\mbox{notation following \eqref{scheme2}})
\\
 &=& \lim_{m\rightarrow\infty}
  \sum_{j=1}^m  \int_{(\dL2)^2} \frac{(  v_1, \Delta w_j)^2}{\|\nabla
  w_j\|^2}  \,d\mu_k(\bV)
\\
&& (\mbox{invariance of $\mu_k$ under $\bSk$ again})
\\
&=&
 \int_{(\dL2)^2} \|\nabla v_1 \|^2 \,d\mu_k(\bV),
\end{eqnarray*}
as well as 
\begin{eqnarray*}
\int_{(\dL2)^2} \|\nabla v_1 \|^2 \,d\mu_0(\bV)
&=& \lim_{m\rightarrow\infty}
  \sum_{j=1}^m  \int_{(\dL2)^2} \frac{(\nabla v_1, \nabla w_j)^2}{\|\nabla
  w_j\|^2}  \,d\mu_0(\bV)
\\
&=& \lim_{m\rightarrow\infty}\lim_{k\rightarrow 0}
  \sum_{j=1}^m  \int_{(\dL2)^2} \frac{(\nabla v_1, \nabla w_j)^2}{\|\nabla
  w_j\|^2}  \,d\mu_k(\bV)
\\
&\le& \liminf_{k\rightarrow 0}
  \sum_{j=1}^\infty  \int_{(\dL2)^2} \frac{(\nabla v_1, \nabla w_j)^2}{\|\nabla
  w_j\|^2}  \,d\mu_k(\bV)
\\
&=& \liminf_{k\rightarrow 0} \int_{(\dL2)^2} \|\nabla v_1\|^2 \,d\mu_k(\bV),
\end{eqnarray*}
together with the fact that on the support of $\mu_k$
\begin{eqnarray*}
&&|b(2\psi_2-\psi_1, 2v_2-v_1, (\bSk(\bV))_2) - b(\psi_1, v_1, v_1)|
\\
&\le& 
|b(2\psi_2-\psi_1, 2v_2-v_1, (\bSk(\bV))_2-v_1)|
+|b(2\psi_2-\psi_1, 2(v_2-v_1), v_1)|
+|b(2(\psi_2-\psi_1), v_1, v_1)|
\\
&\le&
5C_w \|\nabla\bV\|^2 (\| (\bSk(\bV))_2-v_1)\|+\|v_2-v_1\|)
\\
&\le&
\mathcal{O}(k)
\end{eqnarray*}
thanks to Lemma \ref{diff} and the Wente type estimates Lemma \ref{wente}.

This completes the proof of the energy type inequality (no.3 in
the definition) for the limit probability measure $\mathcal{P}_1^*\mu_0$. Therefore
we conclude that the limit $\mathcal{P}_1^*\mu_0$ must be an invariant measure of
the 2D Navier-Stokes system.

\end{proof}

\medskip

\begin{remark}[attractor convergence] Since the most important part of the global attractor must lie in the support of at least one of the invariant measures of the system, the convergence of invariant measures that we derived in Theorem \ref{conv-IM} already implies the convergence of an important part of the global attractor. Convergence of the whole global attractor (in an upper semi-continuous fashion)  can be discussed as well following the "monoid" framework introduce by Hill and Suli \cite{HillSuli1995}. One would work with the so-called "monoid" $[1,1]^T S(t) [-\frac12,\frac32]: (\dL2)^2 \rightarrow (\dL2)^2$. This is not a semi-group since it does not satisfy the requirement that it is the identity map at time zero.
We would like to point out that the convergence of the global attractors bears no implication on the convergence of invariant measures since knowing the support of the measure provides little info on the measure itself.
\end{remark}

  \section{Fully discretized case}
  
  The analysis above can be carried over to the fully discretized case with either Galerkin Fourier spectral approximation or collocation Fourier spectral approximation (see the classical books by  Gottlieb and Orszag \cite{GO1977}, or Canuto, Hussaini, Quarteroni and Zang \cite{CHQZ1988}, or Peyret \cite{P2002} for more on spectral methods as well as their applications in computational  fluid dynamics.).

  \subsection{Galerkin Fourier spectral approximation}

This subsection is devoted to the long time stability of the following Galerkin Fourier spectral in space and BDF2-AB2 in time approximation of the two dimensional Navier--Stokes equations
  \be
    \frac{3\omega_N^{n+1}-4\omega_N^n+\om_N^{n-1}}{2k} + P_N(\nabla^\perp(2\psi_N^n-\psi_N^{n-1})\cdot\nabla(2\omega_N^n-\om_N^{n-1})) - \nu\Delta\omega_N^{n+1} = P_N(f^{n+1}).
\label{Galerkin_scheme}
   \ee
 where $\omega_N^n, \psi_N^n  \in \mathcal{P}_N$,  $\mathcal{P}_N = \{\mbox{all\ trigonometric\ functions\ on\ $\Omega$ \ with \ frequency\  in\  each\  direction\  at\ most\  $N$}\}$, $P_N$ is defined as the orthogonal projection from $\dot{L}^2(\Omega)$ onto $\mathcal{P}_N$.

Just like  the semi-discrete scheme \eqref{scheme}, we can  show that  the scheme (\ref{Galerkin_scheme})  is uniformly bounded
in various spaces, provided that the time step is sufficiently small. More precisely, we have the following:

 \begin{stmt}[uniform in time bounds on the Galerkin Fourier BDF2AB2 scheme] \label{t:bdGF}
 Let $\omega_N^n$ be the solution of the numerical scheme (\ref{Galerkin_scheme}) and let  $f \in L^\infty(\Real_+;\dH1)$. Then the same estimates as those stated in Theorems \ref{t:bdh}, \ref{t:bdv} hold under the same time-step restriction \eqref{5a}. In particular, the Galerkin Fourier spectral scheme \eqref{Galerkin_scheme} is long time stable in the sense that the solution is uniformly bounded in time and in truncation wave number $N$ using either the $\dL2$ or $\dH1$ or $\dot{H}^2_{per}$ norm.
\end{stmt}

The proof is essentially a verbatim copy of that of the semi-discrete in time case and we leave it to the interested reader.

\subsection{Fourier collocation approximation}

We now consider Fourier collocation spectral approximation of the semi-discrete scheme \eqref{scheme}. The collocation spectral approximation may be desirable since fast Fourier transform can be utilized to evaluate the nonlinear advection term efficiently.  Here we follow a recent work by Gottlieb, Tone, Wang, Wang and Wirosoetisno \cite{GTWWW2011}.

\subsubsection{Fourier collocation notation} 

In order to present our fully discrete scheme with collocation spectral approximation, we need to introduce various discrete differential operators on these Fourier collocation spectral space.

For a fixed positive integer $N=N_x=N_y$, we define the mesh size $h=\frac{2\pi}{N}=h_x=h_y$, together with the numerical grid points $(x_i, y_j)$,
with $x_i = i h$, $y_j=jh$, $0 \le i , j \le N$.

Let $f$ be a periodic function over the given 2-D numerical grid with its discrete Fourier expansion  given by
\begin{equation}
  f_{i,j} = \sum_{k_1,l_1=-[N/2]}^{[N/2]}
   \hat{f}_{k_1,l_1} \exp \left(  k_1  {\rm i} x_i \right)
   \exp \left(  l_1  {\rm i} y_j \right).
   \label{spectral-coll-1}
\end{equation}
Then its collocation Fourier spectral approximations to first and second order
partial derivatives are given by
\begin{eqnarray}
  \left( {\cal D}_{Nx} f \right)_{i,j} = \sum_{k_1,l_1=-[N/2]}^{[N/2]}
   \left( k_1  {\rm i} \right) \hat{f}_{k_1,l_1}
   \exp \left(  {\rm i} ( k_1 x_i + l_1 y_j ) \right) ,   \label{spectral-coll-2-1}
\\
  \left( {\cal D}_{Nx}^2 f \right)_{i,j} = \sum_{k_1,l_1=-[N/2]}^{[N/2]}
   \left( - k_1^2 \right) \hat{f}_{k_1,l_1}
   \exp \left(  {\rm i} ( k_1 x_i + l_1 y_j) \right) , \label{spectral-coll-2-3}
\end{eqnarray}
with ${\cal D}_{Ny}, {\cal D}_{Ny}^2$ defined analogously.
In turn, the discrete Laplacian, gradient and divergence operators are defined as
\begin{eqnarray}
  \Delta_N f =  \left( {\cal D}_{Nx}^2  + {\cal D}_{Ny}^2 \right) f , 
\quad 
  \nabla_N f = \left(  \begin{array}{c}
  {\cal D}_{Nx} f  \\
  {\cal D}_{Ny} f
  \end{array}  \right)  ,  \quad
  \nabla_N \cdot \left(  \begin{array}{c}
  f _1 \\
  f _2
  \end{array}  \right)  = {\cal D}_{Nx} f_1 + {\cal D}_{Ny} f_2 ,  \label{spectral-coll-3}
\end{eqnarray}
at the point-wise level. It is also straightforward to verify that
\begin{equation}
  \nabla_N \cdot \nabla_N f = \Delta_N f .  \label{spectral-coll-4}
\end{equation}

  We also recall, for any given periodic grid functions $f$ and $g$ over the 2-D numerical grid,
the spectral approximations to the $L^2$ inner product and $L^2$ norm are defined as
\begin{eqnarray}
  \left\| f \right\|_2 = \sqrt{ \left\langle f , f \right\rangle } ,  \quad \mbox{with} \quad
  \left\langle f , g \right\rangle  = h^2 \sum_{i,j=0}^{N -1}   f_{i,j} g_{i,j} .
  \label{spectral-coll-inner product-1}
\end{eqnarray}
Meanwhile, such a discrete $L^2$ inner product can also be viewed in the Fourier space other than in physical space, with the help of Parseval's identity:
\be
   \left\langle f , g \right\rangle
   =   \sum_{k_1,l_1=-[N/2]}^{[N/2]}
   \hat{f}_{k_1,l_1}  \overline{\hat{g}_{k_1,l_1}}
   =   \sum_{k_1,l_1=-[N/2]}^{[N/2]}
   \hat{g}_{k_1,l_1}  \overline{\hat{f}_{k_1,l_1}}  ,
   \label{spectral-coll-inner product-2}
\ee
in which $\hat{f}_{k_1,l_1}$, $\hat{g}_{k_1,l_1}$ are the Fourier coefficients of the grid functions $f$ and  $g$ in the expansion as in (\ref{spectral-coll-1}). Furthermore, a detailed calculation shows that the following formulas of integration by parts are also valid at the discrete level:
\begin{eqnarray}
  \left\langle f ,  \nabla_N \cdot \left(  \begin{array}{c}
  g_1 \\
  g_2
  \end{array}  \right)  \right\rangle  = - \left\langle \nabla_N f ,  \left(  \begin{array}{c}
  g_1 \\
  g_2
  \end{array}  \right)  \right\rangle ,   \qquad
  \left\langle f ,  \Delta_N  g  \right\rangle
  = - \left\langle \nabla_N f ,  \nabla_N g   \right\rangle  .
  \label{spectral-coll-inner product-3}
\end{eqnarray}

Discrete Sobolev spaces (with fractional order) on the 2D numerical grid can be defined utilizing the same idea as in the continuous case with the help of the discrete Laplacian operator.

\medskip


It is well-known that the existence of aliasing error in the nonlinear term
poses a serious challenge in the numerical analysis of Fourier collocation
spectral scheme. Next, we recall a periodic extension of a grid function and a Fourier
collocation interpolation operator is introduced (see for instance \cite{GTWWW2011}).
\begin{definition}
  For any periodic grid function $f$ defined over a uniform 2-D numerical grid,
we denote $\f_N$  its periodic extension. More specifically,
assume that the grid function $f$ has a discrete Fourier expansion in the form of
\begin{equation}
  f_{i,j}  = \sum_{k_1,l_1=-[N/2]}^{[N/2]}
   \hat{f}_{k_1,l_1}
     \exp \left( {\rm i} ( k_1 x_i + l_1 y_j) \right) ,
    \label{spectral-coll-projection-1}
\end{equation}
its periodic extension (projection) into $P^N$ is then defined as 
\begin{equation}
   \f_N (\x)  = \sum_{k_1,l_1=-[N/2]}^{[N/2]}
   \hat{f}_{k_1,l_1}
     \exp \left( {\rm i} ( k_1 x + l_1 y ) \right) .
    \label{spectral-coll-projection-2}
\end{equation}
Moreover, for any periodic continuous function $\f$, which may contain larger
wave length, we define its collocation interpolation operator as
\begin{eqnarray}
  & &
  \f_{i,j}  = \sum_{k_1,l_1=-[N/2]}^{[N/2]}
   (\hat{f}_c)_{k_1,l_1}
     \exp \left( {\rm i} ( k_1 x_i + l_1 y_j ) \right) ,   \nonumber
\\
  & &
    P_c^N \f_N (\x)  = \sum_{k_1,l_1=-[N/2]}^{[N/2]}
   (\hat{f}_c)_{k_1,l_1}
     \exp \left( {\rm i} ( k_1 x + l_1 y ) \right) .
    \label{spectral-coll-projection-3}
\end{eqnarray}
Note that $\hat{f}_c$ may not be the Fourier coefficients of $\f$, due to the aliasing
errors.
\end{definition}

\subsubsection{The collocation Fourier spectral BDF2 AB2 scheme}

We are now ready to present a collocation Fourier spectral approximation in space of the 2nd order BDF2AB2 scheme \eqref{scheme} as follows:
  \begin{eqnarray}
&&    \frac{3\omega_N^{n+1}-4\omega_N^n+\om_N^{n-1}}{2k} + \frac12(\nabla_N^\perp(2\psi_N^n-\psi_N^{n-1})\cdot\nabla(2\omega_N^n-\om_N^{n-1}) 
+\nabla_N\cdot(\nabla_N^\perp(2\psi_N^n-\psi_N^{n-1})(2\om_N^n-\om^{n-1})))
\nonumber \\
&=&  \nu\Delta\omega_N^{n+1} + f_N^{n+1},
\label{collo-scheme}
\\
&& -\Delta_n\psi^j = \om_N^j, j=n-1, n.
   \end{eqnarray}

Note that the nonlinear term is an explicit 2nd order in time spectral approximation to alternative formulation of the nonlinear advection term as  $\frac12 \left( \nabla^\perp\psi \nabla \omega + \nabla \cdot \left( \nabla^\perp\psi \omega \right) \right)$ at time step $n+1$. This alternative formulation of the nonlinear term is due to Temam \cite{Temam1966}.

A straightforward application of integration by parts formula
(\ref{spectral-coll-inner product-3}) gives
\begin{eqnarray}
  \left\langle  \omega ,  \u \DOT \nabla_N \omega
  + \nabla_N \cdot \left( \u \omega \right)   \right\rangle
   = \left\langle  \omega ,  \u \DOT \nabla_N \omega  \right\rangle
  - \left\langle \nabla_N \omega , \u \omega   \right\rangle
  = 0  .  \label{scheme-coll-1st-5}
\end{eqnarray}

This orthogonality property is important in deriving uniform in time and in mesh size estimates on the solutions to the scheme \eqref{collo-scheme}.
Indeed, this can be combined with a inductive and bootstrap argument starting with the assumption that there exists a constant $C_\delta$ such that  
$
  \| \omega_N^j \|_{H_h^\delta} \le C_\delta ,  
$
for some $\delta > 0$ at time step $j=n-1,n$ to deduce uniform in time bounds. Detail for a first order in time scheme that treat the viscous term implicitly while the advection term explicitly can be found in the recent work by Gottlieb,Tone,  Wang, Wang and Wirosoetisno \cite{GTWWW2011}. We leave out the detail for the verification of the 2nd order in time scheme \eqref{collo-scheme} due to space limitation as a future work.

  \section{Conclusion and remarks}
  
  We have shown that the second order in time scheme \eqref{scheme} is long time global stable in the sense that the solutions remain bounded in $\dL2$ and are asymptotically bounded in $\dH1$ and $\dot{H}^2$ provided that the time step is sufficiently small. This is a very efficient scheme since only a Poisson type solver is needed at each time step.
  This scheme is advantages over the fully explicit ones in terms of stability, and over the fully implicit one in terms of efficiency and unique solvability (the unique solvability of fully implicit scheme requires small time-step as well). 
  We have also demonstrated that this scheme can be viewed as a dissipative discrete in time dynamical system on a product space as defined in \eqref{scheme2}. Moreover, the marginal distributions of the invariant measures of the discrete dynamical system \eqref{scheme2} converge to invariant measures of the Navier-Stokes system \eqref{NSE} at vanishing time-step.  
  Hence the long time statistics of the numerical scheme \eqref{scheme} converge to those of the original 2D NSE \eqref{NSE}.
   Globally long time stable fully discretized versions via spectral Galerkin or spectral collocation are also discussed.
  Therefore, the scheme \eqref{scheme} that we proposed looks very promising as a tool in the numerical study of long time statistical properties of the 2D Navier-Stokes system in a periodic box. Numerical experiments are under way to test physically interesting cases such as  enstrophy transfer  etc.
  
  However, there are many questions that remain to be answered.
  
  First, it is well-known that the set of invariant measures to the NSE \eqref{NSE} contains more than one point at high enough generalized Grashoff number in general situation thanks to classical bifurcation analysis result (since there are multiple steady states in generic case, see for instance \cite{T1983}). Therefore the convergence that we derived is only an upper semi-continuity result. On the other hand, it is generally believed that the physically relevant and observable equilibrium statistical property of the 2D NSE at large enough Reynolds number is unique and consequently they can be characterized via Kraichnan's law etc (see for instance \cite{F1995, FMRT2001, MY1975, C1994} among others). Hence the question of if the numerical scheme that we proposed above \eqref{scheme} can capture this "physically relevant one" becomes an important issue. One could argue that the physically relevant one should be robust under small perturbation and hence any reasonable numerical scheme (such as the one that we proposed in this paper) should be able to capture the robust behavior asymptotically. Yet another approach is to use small random perturbation to bring out the generic behavior of the underlying system since it is known that many dissipative system with appropriate noise would possess a unique invariant measure (see for instance Da Prato and Zabczyk \cite{DZ1996}, E \cite{E2001}, Hairer and Mattingly \cite{HairerMattingly2008} among many others). In this case we then need to study numerical schemes that are able to capture the long time statistical properties of infinite dimensional (since we work on PDE) random (for the noise) dynamical systems. Besides issues that are parallel to the deterministic case, new questions arises such as the form of the noise as well as the magnitude etc. Of course random dynamical systems emerge in many other contexts including the important application of modeling model errors. We refrain from surveying work in this area.
  
  Second, the global in time stability result that is rigorously proved here imposes a time-step restriction \eqref{5a} although it is of the same order (in terms of the dependence on the viscosity) as that for the first order scheme that we investigated earlier \cite{GTWWW2011}.  We believe that some kind of time-step restriction is needed due to the explicit treatment of the advection term.  It would be interesting to find out the best possible time-step restriction for small $\nu$ (we need to set $f$ to be of the order of $\nu$ as well if we are interested in having a meaningful small $\nu$ asymptotics) . In practice, an on-the-fly criterion may be used to ensure the boundedness of the solution within the discrete $L^2$ norm for instance. 
  
  
    Third, we have investigated the periodic boundary and one particular efficient second order in time discretization  only. It would be interesting to study the case with the physically more interesting no-slip boundary condition, as well as other efficient implicit-explicit (IMEX) schemes such as treating the advection term explicitly using classical second order linear multistep approach, Crank-Nicolson leap frog type approach, other spatial discretization, as well as higher order in time schemes. 
    Of course it is also of interest to investigate applicability of such schemes to other systems. 
   We note that globally long time stable schemes do exist according to Gottlieb and Wang \cite{GottliebWang2011}. However, it is not clear if those schemes are able to capture long time statistical properties since the modified energy used to show their stability depends on time-step, and the norm is not equivalent to the standard one.  
   Higher order BDF may not be desirable due to the lack of A-stability \cite{HW2002}.


  \section*{Appendix A:  Wente type estimates}
    We recall here a few Wente type estimates from \cite{GTWWW2011} that are applicable to our doubly periodic setting.
  Original estimate of the Jacobian term (essentially $H^{-1}$ norm) goes back to Wente \cite{Wente1969}.
   $L^2$ norm of the Jacobian in the case with homogeneous Dirichlet boundary condition can be found in Kozono and Taniuchi \cite{KT2000}, as well as Kim \cite{Kim2009}. 

\begin{lm}\label{wente}
  There exists an absolute constant $C_w\ge 1$ such that
  \begin{eqnarray}
   \|\nabla^\perp\psi \cdot \nabla \phi \|_{H^{-1}} &\le&
         C_w \|\psi\|_{H^1} \|\phi\|_{H^1},
     \quad \forall \psi\in \dot{H}^1_{per}, \phi \in \dot{H}^1_{per}(\Omega),
  \\
  \|\nabla^\perp\psi \cdot \nabla \phi \|_{H^{-1}} &\le&
         C_w \|\psi\|_{H^2} \|\phi\|_{L^2},
     \quad \forall \psi\in \dot{H}^2_{per}, \phi \in \dot{L}^2(\Omega),
  \\    
  \|\nabla^\perp\psi \cdot \nabla \phi \|_{L^2} &\le&
         C_w \|\psi\|_{H^2} \|\phi\|_{H^1},
     \quad \forall \psi\in \dot{H}^2_{per}, \phi \in \dot{H}^1_{per}(\Omega) \label{158},
  \\
     \|\nabla^\perp\psi \cdot \nabla\phi \|_{L^2} &\le&
      C_w \|\psi\|_{H^1}\|\phi\|_{H^2},
      \quad \forall \psi\in \dot{H}^1_{per}, \phi \in \dot{H}^2_{per}(\Omega),
      \\
     \|\nabla^\perp\psi \cdot \nabla\phi \|_{H^1} &\le&
      C_w \|\psi\|_{H^2}\|\phi\|_{H^2},
      \quad \forall \psi, \phi \in \dot{H}^2_{per}(\Omega) .
    \end{eqnarray}
\end{lm}
%

\section*{Appendix B: A discrete Gronwall type inequality for two step iterations}
\begin{lm} \label{gronwall}
Let $\{g^n\}$ be a non-negative sequence. Suppose there exist constants  $\varepsilon>0, \beta >0, \lambda \in (0,1)$ such that
\be
  g^{n+1} \le \frac{\lambda}{1+\varepsilon}g^n + \frac{1-\lambda}{1+\varepsilon} g^{n-1} + \frac{\beta\varepsilon}{1+\varepsilon}, \forall n\ge 1. 
 \ee
 Then we have,   for $\gamma = \frac{1+\varepsilon/2}{1+\varepsilon} <1$, and $n\ge 2$,
 \begin{eqnarray}
   g^{n+1} & \le & \gamma \max\{g^n, g^{n-1}, 2\beta\},
   \label{Gstep1}
   \\
   g^{n+1} & \le & \gamma \max\{\gamma^{\lfloor \frac{n-1}{2}\rfloor}g^2, \gamma^{\lfloor \frac{n-1}{2}\rfloor}g^{1}, 2\beta\}. 
   \label{Gfinal}
   \end{eqnarray}
   where ${\lfloor\cdot \rfloor}$ denotes the floor function  (the biggest integer bounded by $\cdot$). 
 \end{lm}
 
 \begin{proof}
   The proof is a straightforward induction based on if $n$ is even or odd and we leave it to the interested reader.
  
 \end{proof}

\section*{Acknowledgement}
This work is supported in part by grants from the National Science
Foundation.


\end{document}